\documentclass[12pt,oneside,reqno,english]{amsart}
\usepackage{lmodern}
\usepackage[T1]{fontenc}
\usepackage[latin9]{inputenc}
\usepackage{geometry}
\usepackage{mathrsfs}
\usepackage{amsbsy}
\usepackage{amstext}
\usepackage{amsthm}
\usepackage{amssymb}
\usepackage{graphicx}
\usepackage{setspace}
\makeatletter
\numberwithin{equation}{section}
\numberwithin{figure}{section}
\theoremstyle{remark}
\newtheorem*{acknowledgement*}{\protect\acknowledgementname}
\theoremstyle{plain}
\newtheorem{thm}{\protect\theoremname}[section]
\theoremstyle{remark}
\newtheorem{rem}[thm]{\protect\remarkname}
\theoremstyle{plain}
\newtheorem*{question*}{\protect\questionname}
\theoremstyle{plain}
\newtheorem{prop}[thm]{\protect\propositionname}
\theoremstyle{plain}
\newtheorem{lem}[thm]{\protect\lemmaname}

\DeclareFontFamily{OMX}{MnSymbolE}{}
\DeclareFontShape{OMX}{MnSymbolE}{m}{n}{
    <-6>  MnSymbolE5
   <6-7>  MnSymbolE6
   <7-8>  MnSymbolE7
   <8-9>  MnSymbolE8
   <9-10> MnSymbolE9
  <10-12> MnSymbolE10
  <12->   MnSymbolE12}{}
\DeclareSymbolFont{mnlargesymbols}{OMX}{MnSymbolE}{m}{n}
\SetSymbolFont{mnlargesymbols}{bold}{OMX}{MnSymbolE}{b}{n}
\DeclareMathDelimiter{\llangle}{\mathopen}{mnlargesymbols}{'164}{mnlargesymbols}{'164}
\DeclareMathDelimiter{\rrangle}{\mathclose}{mnlargesymbols}{'171}{mnlargesymbols}{'171}
\IfFileExists{lmodern.sty}{\usepackage{lmodern}}{}

\makeatother

\usepackage{babel}
\providecommand{\acknowledgementname}{Acknowledgement}
\providecommand{\lemmaname}{Lemma}
\providecommand{\propositionname}{Proposition}
\providecommand{\questionname}{Question}
\providecommand{\remarkname}{Remark}
\providecommand{\theoremname}{Theorem}

\begin{document}

\title{Analysis of Metastable Behavior via Solutions of Poisson Equations}

\author{Insuk Seo}

\address{Department of Mathematical Sciences and R.I.M., Seoul National University
\\
27-212, Gwanak-Ro 1, Gwanak-Gu, Seoul 08826, Republic of Korea.}

\email{insuk.seo@snu.ac.kr}
\begin{abstract}
We herein review the recent progress on the study of metastability
based on the analysis of solutions of Poisson equations related to
the generators of the underlying metastable dynamics. This review
paper is based on the joint work with Claudio Landim \cite{LS3} and
Fraydoun Rezakhanlou \cite{RS}.
\end{abstract}

\maketitle
\begin{acknowledgement*}
This article is written as a proceeding to the 17th International
Symposium \textit{Stochastic Analysis on Large Scale Interacting Systems}
held in the Research Institute for Mathematical Sciences (RIMS), Kyoto
University from November 5--8, 2018. The author would like to thank
the organizers for the invitation to the conference and the kind support
during the stay. The author also thanks Professor Kenkichi Tsunoda
for the invitation to Osaka University before the conference was held,
during which most of the current article was written.

The author is deeply indebted to Professor Claudio Landim, who introduced
the metastability theory to the author and shared numerous valuable
ideas, and to Professor Fraydoun Rezakhanlou who first envisioned
the approach explained herein and shared it with the author.

This work is supported by the National Research Foundation of Korea
(NRF), grant funded by the Korea government (MSIT) (No. 2018R1C1B6006896
and No. 2017R1A5A1015626).
\end{acknowledgement*}

\section{\label{sec1}Quantitative Analysis of Metastable Behavior}

Metastable behavior is a ubiquitous phenomenon exhibited by random
dynamics in a low-temperature regime. To concretely describe this
behavior, we first introduce a classic model known as the \textit{small
random perturbation of dynamical systems} (SRPDS, see \cite{FW3}
for an extensive discussion). For a smooth potential function $U:\mathbb{R}^{d}\rightarrow\mathbb{R}$
and small parameter $\epsilon>0$, we consider a stochastic differential
equation given by

\begin{equation}
d\boldsymbol{y}_{\epsilon}(t)=-\nabla U(\boldsymbol{y}_{\epsilon}(t))dt+\sqrt{2\epsilon}\,d\boldsymbol{w}_{t}\;;\;t\ge0\;,\label{e12}
\end{equation}
where $(\boldsymbol{w}_{t})_{t\ge0}$ represents the standard $d$-dimensional
Brownian motion. If $U$ has several local minima as in Figure \ref{fig1},
the stochastic process $(\boldsymbol{y}_{\epsilon}(t))_{t\ge0}$ exhibits
the metastable behavior. To describe this behavior, we consider the
zero-temperature dynamics described by the following ordinary differential
equation:
\begin{equation}
d\boldsymbol{y}(t)=-\nabla U(\boldsymbol{y}(t))dt\;\;;\;t\ge0\;.\label{e11}
\end{equation}
Then, we can observe that the local minima of the potential function
$U$ is the stable equilibria of the dynamics $(\boldsymbol{y}(t))_{t\ge0}$.
Namely, the dynamics starting at a domain of attraction of a local
minimum $\boldsymbol{m}_{1}$ of $U$ converges to $\boldsymbol{m}_{1}$
exponentially fast. We can regard \eqref{e12} as a small random perturbation
of the deterministic dynamical system \eqref{e11}, provided that
$\epsilon>0$ is small enough; hence, one can expect a similar behavior;
the random process $(\boldsymbol{y}_{\epsilon}(t))_{t\ge0}$ starting
at a neighborhood of the local minimum $\boldsymbol{m}_{1}$ of $U$
will converge to $\boldsymbol{m}_{1}$.

\begin{figure}
\includegraphics[scale=0.2]{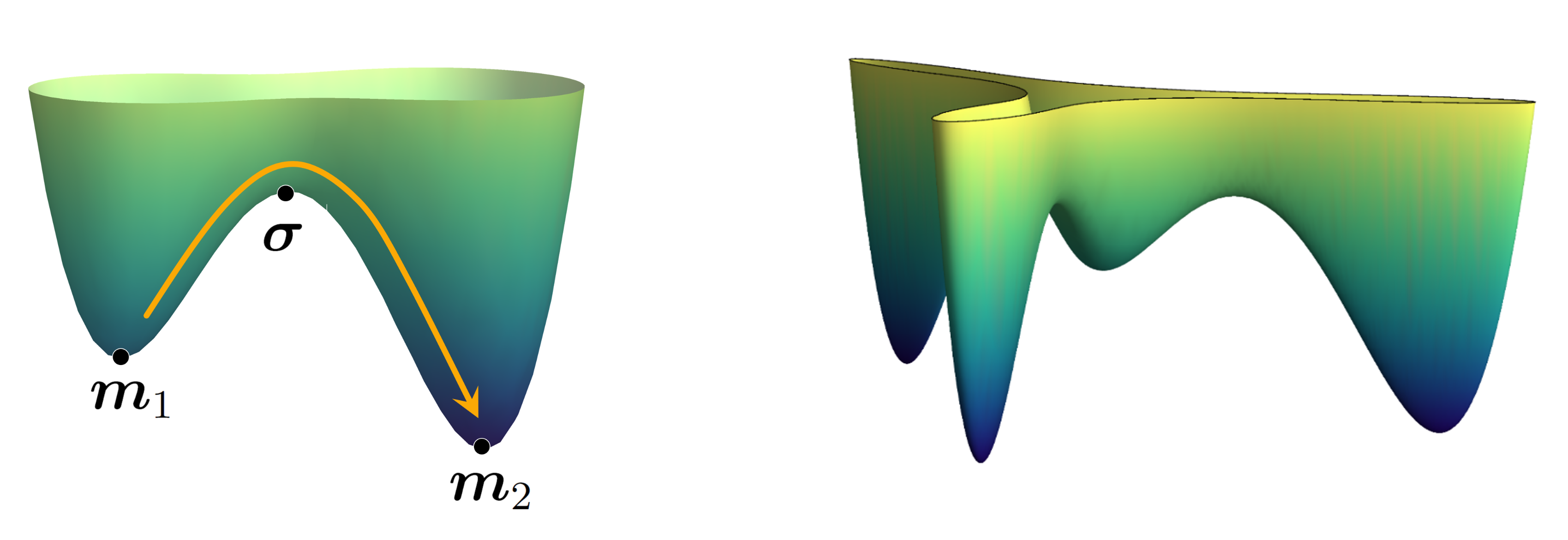}

\caption{Example of potential functions: (Left) Double-well potential, (Right)
General potential with several global minima\label{fig1}}
\end{figure}

This estimates is locally true; however, if we consider the global
picture, a crucial difference arises because of the randomness induced
by the Brownian motion: if we wait for a sufficiently long time, then
the process will move from the neighborhood of $\boldsymbol{m}_{1}$
to that of another local minimum, e.g., $\boldsymbol{m}_{2}$ (see
Figure \ref{fig1}-(Right)). We now call the neighborhoods of each
local minimum as the metastable set, and the transitions among these
sets explained above are called the metastable transition. We expect
that this metastable transition occurs repeatedly in a suitable time
scale, and is an example of metastable behavior. This type of behavior
is exhibited by numerous models, e.g., interacting particle systems
such as the zero-range processes \cite{AGL, BL3, Lan2, Seo1}, simple
inclusion processes \cite{BDG, GRV2}, and ferromagnetic systems such
as the Curie-{}-Weiss model \cite{BBI}, Ising model \cite{BM} and,
Potts model \cite{LS2, NZ}.

To explain the primary questions in this study, we only focus on the
SRPDS \eqref{e12} in this introductory section. In the SRPDS, we
can consider two cases as below, and the primary concerns are slightly
different for each case.
\begin{itemize}
\item Case 1: $U$ contains several local minima but only one global minimum
as in Figure \ref{fig1}-(Left). In this case, the process starting
from any point first stabilizes at a neighborhood of the local minima.
Subsequently, after a sufficiently long time, it performs a metastable
transition toward the neighborhood of the global minimum, and remains
therein for a longer time scale than the metastable transition time.
Therefore, a primary object to be investigated is this metastable
transition time from a local minimum to the global minimum. For instance,
its expected value and asymptotic law are the primary concern. The
robust methodology answering this question in a quantitative manner
is the \textit{potential-theoretic approach }developed by Bovier et
al. in \cite{BEGK2, BEGK1}. This methodology is explained in Section
\ref{s21}.
\item Case 2: $U$ contains multiple global minima as in Figure \ref{fig1}-(Right).
As explained previously, a process starting from a neighborhood of
a minimum will be stabilized in this set. However, after a sufficiently
long time, we observed the metastable transition, and expect that
this transition occurs repeatedly. Therefore, describing such hopping
dynamics in a rigorous manner is an important problem from the mathematical
perspective. In particular, one expects to demonstrate that a scaling
limit of this process converges to a Markov chain among the metastable
sets, in a suitable sense. In particular, if the underlying metastable
dynamics is a Markov process on a finite set, a robust methodology
for proving this scaling limit, which is called \textit{martingale
approach}, has been established by Beltran and Landim in \cite{BL1, BL2}.
We explain this approach in Section \ref{s22}.
\end{itemize}
Our new methodology introduced in Section \ref{sec3} can be regarded
as another approach to analyze Case 2, and provides concrete answers
for some questions that cannot be answered by both of the approaches
explained previously. In particular, we can conduct a rigorous analysis
explained in Case 2, when the underlying dynamics is a diffusion process
that is not a process on a finite set. We refer to \cite{LS3, RS}
for details. In this review paper, instead of focusing on this specific
model, we attempt to deliver our general idea on this new approach
by skipping the technical details. In Section \ref{sec2}, we review
the previous approaches; in Section \ref{sec3}, we explain our alternative
approach.

\section{\label{sec2}Review on Previous Approaches}

We shall consider a family of Markov processes $\{(\boldsymbol{x}_{\epsilon}(t))_{t\ge0}:\epsilon>0\}$
on $\Omega$. The sets $\mathcal{E}_{1},\,\ldots,\,\mathcal{E}_{K}$
represent the metastable set, i.e., the process $\boldsymbol{x}_{\epsilon}(t)$
starting from a point in a set $\mathcal{E}_{i}$ remains in this
set sufficiently long when $\epsilon$ is small enough. Subsequently,
after a sufficiently long time, the process exhibits a transition
to another metastable set. We write $S=\{1,\,2,\,\ldots,\,K\}$ and
let $\mu_{\epsilon}$ be the invariant measure of the process $\boldsymbol{x}_{\epsilon}(t)$.
We denote by $\mathbb{P}_{\boldsymbol{x}}^{\epsilon}$ the law of
process $\boldsymbol{x}_{\epsilon}(t)$ starting from $\boldsymbol{x}\in\Omega$,
and denote by $\mathbb{E}_{\boldsymbol{x}}^{\epsilon}$ the corresponding
expectation.
\begin{rem}
The sets $\Omega$ and $\mathcal{E}_{1},\,\ldots,\,\mathcal{E}_{K}$
may depend on $\epsilon$. Instead of the family of Markov processes
parameterized by $\epsilon$, one can consider the sequence of Markov
processes $\{(\boldsymbol{x}_{N}(t))_{t\ge0}:N\in\mathbb{N}\}$. The
approaches explained below can be applied to this sequence as well.
Instead of the $\epsilon\downarrow0$ limit, we should consider the
$N\uparrow\infty$ limit in this case.
\end{rem}

\subsection{\label{s21}Potential-theoretic approach}

We first consider the question suggested in Case 1 of Section \ref{sec1},
namely, the estimation of the metastable transition time from a metastable
set to others. To formulate this question concretely, we first introduce
some notations. For $\mathcal{A}\subset\Omega$, we denote by $\tau_{\mathcal{A}}$
the hitting time of the set $\mathcal{A}$. We write $\mathcal{E}:=\bigcup_{i\in S}\mathcal{E}_{i}$
and $\breve{\mathcal{E}}_{i}=\mathcal{E}\setminus\mathcal{E}_{i}$.
We pick a point $\boldsymbol{x}_{i}\in\mathcal{E}_{i}$ arbitrarily.
Subsequently, the primary concern is the estimation of the mean transition
time $\mathbb{E}_{\boldsymbol{x}_{i}}[\tau_{\breve{\mathcal{E}}_{i}}]$.
This quantity corresponds to the escape time from a local minimum,
and is crucially related to the mixing time and the spectral gap (see
\cite{BGK}).

It has been observed in \cite{BEGK1} that, if the process $\boldsymbol{x}_{\epsilon}(t)$
is reversible with respect to the invariant measure $\mu_{\epsilon}$,
then the mean transition time is closely related to a potential theoretic
notion known as the capacity. For disjoint subsets $\mathcal{A}$
and $\mathcal{B}$ of $\Omega$, the equilibrium potential between
$\mathcal{A}$ and $\mathcal{B}$ with respect to the process $\boldsymbol{x}_{\epsilon}(t)$
is a function $h_{\mathcal{A},\,\mathcal{B}}^{\epsilon}:\Omega\rightarrow[0,\,1]$
defined by
\[
h_{\mathcal{A},\,\mathcal{B}}^{\epsilon}(\boldsymbol{x})=\mathbb{P}_{\boldsymbol{x}}^{\epsilon}\left[\tau_{\mathcal{A}}<\tau_{\mathcal{B}}\right]\;.
\]
Subsequently, the capacity between $\mathcal{A}$ and $\mathcal{B}$
is defined by
\[
\textup{cap}_{\epsilon}(\mathcal{A},\,\mathcal{B})=\int_{\Omega}h_{\mathcal{A},\,\mathcal{B}}^{\epsilon}(-\mathcal{L}_{\epsilon}h_{\mathcal{A},\,\mathcal{B}}^{\epsilon})d\mu_{\epsilon}\;,
\]
where $\mathcal{L}_{\epsilon}$ is the generator corresponding to
the process $\boldsymbol{x}_{\epsilon}(t)$. The crucial observation
is that, under the circumstances of metastability, the following holds:
\begin{equation}
\mathbb{E}_{\boldsymbol{x}_{i}}[\tau_{\breve{\mathcal{E}}_{i}}]\simeq\frac{\mu_{\epsilon}(\mathcal{E}_{i})}{\textup{cap}_{\epsilon}(\mathcal{E}_{i},\,\breve{\mathcal{E}}_{i})}\;.\label{met}
\end{equation}
The asymptotic limit of $\mu_{\epsilon}(\mathcal{E}_{i})$ of $\epsilon\downarrow0$
is typically not difficult to obtain. The non-trivial part is to estimate
the capacity $\textup{cap}_{\epsilon}(\mathcal{E}_{i},\,\breve{\mathcal{E}}_{i})$.
If the process $\boldsymbol{x}_{\epsilon}(t)$ is reversible with
respect to $\mu_{\epsilon}$, it can be performed via the Dirichlet
and Thomson principles that provide the upper and lower bounds for
the capacity, respectively.

A successful application of this methodology is the Eyring-Kramers
formula for the SRPDS \eqref{e12}. For this model, $\mathcal{E}_{i}$
can be regarded as a small $O(1)$ neighborhood (or $O(\epsilon^{1/2+\alpha})$,
$\alpha>0$, neighborhood) of each local minimum. To deliver the primary
result in a concrete and simple form, let us temporarily assume that
$U$ is a double-well potential as in Figure \ref{fig1}-(Left): two
local minima $\boldsymbol{m}_{1}$ and $\boldsymbol{m}_{2}$ exist,
and a saddle point $\boldsymbol{\sigma}$ exists between them. The
main result of \cite{BEGK1} shows that, if the Hessians $\nabla^{2}U$
at $\boldsymbol{m}_{1},\,\boldsymbol{m}_{2},$ and $\boldsymbol{\sigma}$
are non-degenerate, and if $(\nabla^{2}U)(\boldsymbol{\sigma})$ has
a unique negative eigenvalue $-\lambda_{\boldsymbol{\sigma}}$, then
under some minor technical assumptions, the following holds:
\begin{equation}
\mathbb{E}_{\boldsymbol{m}_{1}}^{\epsilon}[\tau_{\mathcal{E}_{2}}]\,\text{\ensuremath{\simeq}}\,\frac{2\pi}{\lambda_{\sigma}}\,\sqrt{\frac{-\det(\nabla^{2}U)(\boldsymbol{\sigma})}{\det(\nabla^{2}U)(\boldsymbol{m}_{1})}}\,\exp\left\{ \frac{U(\boldsymbol{\sigma})-U(\boldsymbol{m}_{1})}{\epsilon}\right\} \;.\label{e14}
\end{equation}
It has also been verified that the normalized mean transition time
$\tau_{\mathcal{E}_{2}}/\mathbb{E}_{\boldsymbol{m}_{1}}^{\epsilon}[\tau_{\mathcal{E}_{2}}]$
converges to a mean $1$ exponential random variable as $\epsilon\downarrow0$.
We remark that, in a general $U$, a similar but slightly more complicated
expression for the quantity of the form $\mathbb{E}_{\boldsymbol{m}_{1}}^{\epsilon}[\tau_{\breve{\mathcal{E}}_{1}}]$
can be obtained similarly

This approach is extremely robust for the reversible case, as demonstrated
by numerous successful applications for a wide scope of models. Instead
of enumerating these examples, we refer to the monograph \cite{BdH}
for the comprehensive discussions of this stream of studies.

A shortcoming of this original method is that its application is limited
to the case when the process $\boldsymbol{x}_{\epsilon}(t)$ is reversible.
Recently, studies of the non-reversible case has shown significant
improvements. First, Beltran and Landim \cite{BL1, BL2} found a formula
corresponding to \eqref{met} that holds without the reversibility
assumption. This formula is fairly similar to \eqref{met}, and hence
the sharp estimation of capacity is required as well. This is another
difficulty because the classical Dirichlet and Thomson principles
hold only for the reversible case. However, recently, two variational
principles generalizing these principles to the non-reversible case
has been obtained. Gaudilli\`ere and Landim in \cite{GL} found a
generalization of the Dirichlet principle, and Slowik in \cite{Slo}
found that of the Thomson principle. It was developed for the discrete
Markov process setting, but had been generalized to the continuous
diffusion setting in \cite{LMS} as well. These principles are more
difficult to use than the case of reversible counterparts, because
they are double variational principles in complicated spaces. However,
Landim in \cite{Lan2} used this principle creatively to investigate
the metastable behavior of the non-reversible zero-range process.
It was the first study that presented quantitative results in the
study of the metastable behavior of non-reversible processes. More
recently, a manual for using this package of new machinery has been
developed in \cite{LS1} by Landim and the author of this article.
This manual has been used in \cite{LS2}, \cite{Seo1}, and \cite{LMS}
to perform the quantitative analysis on metastability.

\subsection{\label{s22}Martingale approach}

Although the potential-theoretic approach of metastability is a strong
method for analyzing the metastable random process, it cannot answer
the question suggested in Case 2 of Section \ref{sec1}. More precisely,
if there are multiple metastable sets of the same depth, we can expect
that a properly defined rescaled process should behave like a Markov
chain on these metastable sets, but the sharp asymptotics obtained
by the potential-theoretic approach cannot deduce this type of result.
In this subsection, we explain the martingale approach established
by Beltran and Landim in \cite{BL1, BL2, Lan1}. To explain this approach,
we start by introducing several notations.

All the notations defined above are maintained. Recall that there
are $K$ metastable sets $\mathcal{E}_{1},\,\ldots,\,\mathcal{E}_{K}$
for the process $\boldsymbol{x}_{\epsilon}(t)$. We define $S=\{1,\,2,\,\ldots,\,K\}$.
Now we shall define a so-called trace process of $\boldsymbol{x}_{\epsilon}(t)$
on $\mathcal{E}$. Stating heuristically, this is a process obtained
from $\boldsymbol{x}_{\epsilon}(t)$ by turning off the clock when
the process is not in $\mathcal{E}$. To define this object rigorously,
we define the following:
\[
T_{\epsilon}(t)=\int_{0}^{t}\mathbf{1}\{\boldsymbol{x}_{\epsilon}(s)\in\mathcal{E}\}\,ds\;\;;\;t\in[0,\,\infty)\;.
\]
It measures the amount of time that the process $\boldsymbol{x}_{\epsilon}(\cdot)$
has spent on the set $\mathcal{E}$ up to time $t$. Subsequently,
we define the generalized inverse of $T_{\epsilon}(\cdot)$ as
\begin{equation}
S_{\epsilon}(t)=\sup\left\{ s:T_{\epsilon}(s)\le t\right\} \;.\label{Se}
\end{equation}
Finally, the trace process $(\overline{\boldsymbol{x}}_{\epsilon}(t))_{t\ge0}$
is defined by
\[
\overline{\boldsymbol{x}}_{\epsilon}(t)=\boldsymbol{x}_{\epsilon}(S_{\epsilon}(t))\;.
\]
One can check that this process is a Markov process on $\mathcal{E}$,
with possible long jumps along the boundary $\partial\mathcal{E}=\cup_{i\in S}\partial\mathcal{E}_{i}$.
We denote by $\chi_{\mathcal{A}}:\Omega\rightarrow\{0,\,1\}$ the
indicator function on $\mathcal{A}\subset\Omega$, and define the
projection function $\Psi:\mathcal{E}\rightarrow\{1,\,2,\,\ldots,\,K\}$
as
\[
\Psi(\boldsymbol{x})=\sum_{i=1}^{K}i\,\chi_{\mathcal{E}_{i}}(\boldsymbol{x})\;.
\]
Finally, we define the projected process as $\mathbf{x}_{\epsilon}(t)=\Psi(\overline{\boldsymbol{x}}_{\epsilon}(t))$
which is a random process on $S$. With this package of notations,
our primary question can be stated as follows.
\begin{question*}
Can we prove that a scaling limit of the process $(\mathbf{x}_{\epsilon}(t))_{t\ge0}$
converges to a Markov chain $(\mathbf{x}(t))_{t\ge0}$ on $S$?
\end{question*}
To answer this question, it would be ideal if we can determine a priori
prediction of the correct time scale $\theta_{\epsilon}$ and of the
candidate for the limiting Markov chain $\mathbf{x}(t)$. We would
like to stress that these can be inferred from the results for the
mean transition time obtained by the potential-theoretic approach
explained in the previous subsection.

We denote by $\mathbb{P}_{\pi}^{\epsilon}$ the law of the Markov
process $\boldsymbol{x}_{\epsilon}(t)$ with a starting measure $\pi$
on $\Omega$, and by $\mathbf{Q}_{i}$ the law of Markov chain $\mathbf{x}(t)$
starting at $i\in\mathbf{S}$. Finally, we denote by $\mathbf{Q}_{\pi}^{\epsilon}$
the law of the rescaled projected process $\big(\mathbf{x}_{\epsilon}(\theta_{\epsilon}t)\big)_{t\ge0}$
under $\mathbb{P}_{\pi}$. The main theorem can be formulated as follows:
\begin{thm}
\label{tmain}Fix $i\in S$. For $\epsilon>0$, let $\pi_{\epsilon}$
be a probability measure on $\Omega$ concentrated on $\mathcal{E}_{i}$.
Then, $\mathbf{Q}_{\pi_{\epsilon}}^{\epsilon}$ converges to $\mathbf{Q}_{i}$
as $\epsilon$ tends to $0$.
\end{thm}

The martingale approach established by Beltran and Landim in \cite{BL1, BL2, Lan1}
provides a robust methodology to prove this theorem, especially when
the process $\boldsymbol{x}_{\epsilon}(t)$ is a Markov process on
a discrete set. This approach reduces the proof of Theorem \ref{tmain}
to the estimates of the so-called mean-jump rate between valleys.
To state this more precisely, denote by $\overline{r}_{\epsilon}:\mathcal{E}\times\mathcal{E}\rightarrow\mathbb{R}$
the jump rate of the trace process $\overline{\boldsymbol{x}}_{\epsilon}(t)$,
which is a Markov process on $\mathcal{E}$. Subsequently, the mean-jump
rate between two metastable sets $\mathcal{E}_{i}$ and $\mathcal{E}_{j}$
is defined as
\begin{equation}
\mathbf{r}_{\epsilon}(i,\,j)=\frac{1}{\mu_{\epsilon}(\mathcal{E}_{i})}\sum_{\boldsymbol{x}\in\mathcal{E}_{i},\,\boldsymbol{y}\in\mathcal{E}_{j}}\mu_{\epsilon}(\boldsymbol{x})\overline{r}_{\epsilon}(\boldsymbol{x},\,\boldsymbol{y})\;.\label{mjr}
\end{equation}
Denote by $\mathbf{r}:S\times S\rightarrow\mathbb{R}$ the jump rate
of the Markov chain $(\mathbf{x}(t))_{t\ge0}$. The main result of
Beltran and Landim can be summarized as follows: up to several technical
requirements, proving
\begin{equation}
\lim_{\epsilon\rightarrow0}\theta_{\epsilon}\mathbf{r}_{\epsilon}(i,\,j)=\mathbf{r}(i,\,j)\label{lim}
\end{equation}
is enough to demonstrate Theorem \ref{tmain}. Such an implication
has been verified by relating several martingale problems creatively.
Although the rate $\overline{r}_{\epsilon}$, i.e., the jump rate
of the trace process, is not an easy notion to manage, Beltran and
Landim observed that its weighted average $\mathbf{r}_{\epsilon}$
has a rather simple expression in terms of the potential-theoretic
notions, such as the capacity. For instance, if the process $\boldsymbol{x}_{\epsilon}(t)$
is reversible, we have
\begin{equation}
\mathbf{r}_{\epsilon}(i,\,j)=\frac{1}{2}\left[\textup{cap}_{\epsilon}(\mathcal{E}_{i},\,\breve{\mathcal{E}}_{i})+\textup{cap}_{\epsilon}(\mathcal{E}_{j},\,\breve{\mathcal{E}}_{j})-\textup{cap}_{\epsilon}(\mathcal{E}_{i}\cup\mathcal{E}_{j},\,\breve{\mathcal{E}}_{i}\cap\breve{\mathcal{E}}_{j})\right]\;.\label{rij}
\end{equation}
Thus, the analysis of metastable behavior can be performed by estimating
the capacity, similar as before. This technology has been applied
to various reversible models such as the zero-range process \cite{BL3},
the simple inclusion process \cite{BDG}, and random walks in a potential
field \cite{LMT}.

Furthermore, the result of Beltran and Landim in \cite{BL2} indicates
that the implication from \eqref{lim} to Theorem \ref{tmain} also
holds for the non-reversible case. The difficulty in the non-reversible
case is that the formula \eqref{rij} for $\mathbf{r}_{\epsilon}$
is no longer valid. Meanwhile, a rather complicated formula for $\mathbf{r}_{\epsilon}$
in terms of the so-called collapsed Markov chain has been found in
\cite{BL2} and \cite{Lan2}. Based on this, Landim in \cite{Lan2}
first established an analysis of the non-reversible metastable process
by analyzing the totally asymmetric zero-range process. Subseqeuntly,
a robust way to use this complicated formula to deduce \eqref{lim}
is established in \cite{LS1}. Recently, this has been applied to
various non-reversible models in \cite{LS1, LS2, Seo1}. We refer
to the review paper \cite{Lan4} by Landim for the comprehensive discussion
of this topic.

Before concluding this review section, we have to emphasize that the
martingale approach has not yet been successfully applied to metastable
diffusion processes such as the SRPDS \eqref{e12}. Because in this
case the trace process becomes a diffusion process on disconnected
set $\mathcal{E}$ with a long jump along $\partial\mathcal{E}$,
which is not a conventional object, the mean-jump rate such as \eqref{mjr}
is almost impossible to define. Our new methodology (explained in
the next section) can be regarded as an entirely different approach
to metastability, and can be applied to continuous models such as
the SRPDS that cannot be answered by the martingale approach.

\section{\label{sec3}Approach via Poisson Equations}

We recall all the notations from the previous section. In this section,
we explain a new approach developed in \cite{LS3, RS}. This approach
provides an alternative method to prove Theorem \ref{tmain} by analyzing
a Poisson equation, instead of estimating the potential-theoretic
notions such as capacity.

We shall assume that we have predictions of correct time scale $\theta_{\epsilon}$
as well as the limiting Markov chain $\mathbf{x}(t)$, as mentioned
before. In addition, we assume that
\begin{equation}
\mu_{\epsilon}(\mathcal{E}_{i})=(1+o_{\epsilon}(1))\,\nu(i)\;\;;\;\text{for all }i\in\{1,\,\dots,\,K\}\;,\label{inv1}
\end{equation}
where $\nu$ is a measure on $S=\{1,\,\cdots,\,K\}$ satisfying
\begin{equation}
\sum_{i=1}^{K}\nu(i)=1\;,\label{inv2}
\end{equation}
Therefore, using \eqref{inv1} and \eqref{inv2}, we can verify that
$\nu$ is the invariant measure of to the Markov chain $\mathbf{x}(t)$,
provided that Theorem \ref{tmain} is correct.

Denote by $\mathcal{L}_{\epsilon}$ and $\mathbf{L}$ the generators
corresponding to the Markov processes $\boldsymbol{x}_{\epsilon}(t)$
and $\mathbf{x}(t)$, respectively. Define $\mathbf{a}_{\epsilon}=(\mathbf{a}_{\epsilon}(i):i\in S)\in\mathbb{R}^{S}$
by
\[
\mathbf{a}_{\epsilon}(i)=\frac{\nu(i)}{\mu_{\epsilon}(\mathcal{E}_{i})}\;\;;\;i\in S\;.
\]
By \eqref{inv1}, we have
\begin{equation}
\mathbf{a}_{\epsilon}(i)=1+o_{\epsilon}(1)\text{ for all }i\in S\;.\label{aei}
\end{equation}
The following proposition is the primary step in our new approach.
\begin{prop}
\label{poisson}For all $\mathbf{f}:S\rightarrow\mathbb{R}$, there
exists a bounded function $\phi_{\epsilon}=\phi_{\epsilon}^{\mathbf{f}}:\Omega\rightarrow\mathbb{R}$
satisfying the following conditions simultaneously:
\begin{enumerate}
\item The function $\phi_{\epsilon}$ satisfies the following equation:
\begin{equation}
\mathcal{\mathscr{\mathcal{L}}_{\epsilon}}\phi_{\epsilon}=\theta_{\epsilon}^{-1}\sum_{i\in S}\mathbf{a}_{\epsilon}(i)\,(\mathbf{L}\mathbf{f})(i)\,\chi_{\mathcal{E}_{i}}\;.\label{poi1}
\end{equation}
\item For all $i\in S$, the following holds:
\begin{equation}
\lim_{\epsilon\rightarrow0}\sup_{x\in\mathcal{E}_{i}}\left|\phi_{\epsilon}(x)-\mathbf{f}(i)\right|=0\;.\label{poi2}
\end{equation}
\end{enumerate}
\end{prop}

We observed that finding a test function $\phi_{\epsilon}^{\mathbf{f}}$
explained in the last proposition implies Theorem \ref{tmain}, up
to several minor technical issues. Hence, in the remaining part of
the current section, we explain the model-independent proof of Theorem
\ref{tmain} assuming Proposition \ref{poisson}. The proof of Proposition
\ref{poisson} it the only model-dependent part. We explain a general
idea for this part in Section \ref{sec33}.

In general, two ingredients are required to complete the proof of
the limit theorem such as Theorem \ref{tmain}: the tightness of family
$(\mathbf{Q}_{\pi_{\epsilon}}^{\epsilon})_{\epsilon>0}$, and the
identification of limit points of this family. These two crucial ingredients
are proven in Sections \ref{sec22} and \ref{sec23}, respectively.

\subsection{\label{sec22}Tightness}

For convenience, we fix $i\in S$ and the sequence of the family of
probability measures $(\pi_{\epsilon})_{\epsilon>0}$ concentrated
on $\mathcal{E}_{i}$ throughout this subsection. The tightness result
required in our context can be stated as the following proposition.
\begin{prop}
\label{tight}The family $\{\mathbf{Q}_{\pi_{\epsilon}}^{\epsilon}:\epsilon\in(0,\,1]\}$
is tight on $D([0,\,\infty),\,S)$. Furthermore, every limit points
$\mathbf{Q}^{*}$ of this family, as $\epsilon$ tends to $0$, satisfy
\[
\mathbf{Q}^{*}(\mathbf{x}(0)=i)=1\text{\;\;and\;\;}\mathbf{Q}^{*}(\mathbf{x}(t)\neq\mathbf{x}(t-))=0\;\;\text{for all }t>0\;.
\]
\end{prop}

It is shown in \cite[Sections 7, 8]{LS3} that the proof of this tightness
result is based entirely on the two estimates stated in Lemmas \ref{tight2} and \ref{p34}.
The former verifies that a transition from a metastable set to another
one cannot occur in a scale shorter than $\theta_{\epsilon}$, while
the latter demonstrates that in the course of the metastable transition,
the process does not spend much time outside the metastable sets,
namely $\Omega\setminus\mathcal{E}$. We start by proving the former
lemma, whose proof depends solely on Proposition \ref{poisson}. We
simply write $\mathbb{P}_{\boldsymbol{x}}^{\epsilon}:=\mathbb{P}_{\delta_{\boldsymbol{x}}}^{\epsilon}$,
where $\delta_{\boldsymbol{x}}$ is a Dirac delta measure at $\boldsymbol{x}\in\Omega$.
\begin{lem}
\label{tight2}It holds that
\begin{equation}
\lim_{a\rightarrow0}\limsup_{\epsilon\rightarrow0}\sup_{\boldsymbol{x}\in\mathcal{E}_{i}}\mathbb{P}_{\boldsymbol{x}}^{\epsilon}\left[\tau_{\breve{\mathcal{E}}_{i}}\le a\theta_{\epsilon}\right]=0\;.\label{ep471}
\end{equation}
\end{lem}

\begin{proof}
Consider a function $\mathbf{f}:S\rightarrow\mathbb{R}$ defined by
\[
\mathbf{f}(j)=\begin{cases}
0 & \text{if }j=i\\
1 & \text{otherwise\;.}
\end{cases}
\]
Denote by $\phi_{\epsilon}=\phi_{\epsilon}^{\mathbf{f}}$ the function
that we obtain in Proposition \ref{poisson} with respect to this
function $\mathbf{f}$. For $\boldsymbol{x}\in\mathcal{E}_{i}$, by
It\"{o}'s formula and \eqref{poi1}, one can deduce that
\begin{align}
 & \mathbb{E}_{\boldsymbol{x}}^{\epsilon}\left[\phi_{\epsilon}(\boldsymbol{x}_{\epsilon}(a\theta_{\epsilon}\wedge\tau_{\mathcal{E}\setminus\mathcal{E}_{i}}))\right]-\phi_{\epsilon}(\boldsymbol{x})\nonumber \\
 & \quad=\sum_{i\in S}\mathbb{E}_{\boldsymbol{x}}^{\epsilon}\left[\int_{0}^{a\theta_{\epsilon}\wedge\tau_{\mathcal{E}\setminus\mathcal{E}_{i}}}\theta_{\epsilon}^{-1}\mathbf{a}_{\epsilon}(i)\,(\mathbf{L}\mathbf{f})(i)\,\chi_{\mathcal{E}_{i}}(\boldsymbol{x}_{\epsilon}(s))ds\right]\;.\label{et1}
\end{align}
It is noteworthy that for some constant $C>0$,
\begin{equation}
\left|\int_{0}^{a\theta_{\epsilon}\wedge\tau_{\mathcal{E}\setminus\mathcal{E}_{i}}}\theta_{\epsilon}^{-1}\mathbf{a}_{\epsilon}(i)\,(\mathbf{L}\mathbf{f})(i)\,\chi_{\mathcal{E}_{i}}(\boldsymbol{x}_{\epsilon}(s))ds\right|\le a\theta_{\epsilon}\cdot\theta_{\epsilon}^{-1}|\mathbf{a}_{\epsilon}(i)\,(\mathbf{L}\mathbf{f})(i)|\le Ca\;.\label{et2}
\end{equation}
Moreover, since $\boldsymbol{x}\in\mathcal{E}_{i}$, it follows from
\eqref{poi2} that
\begin{equation}
|\phi_{\epsilon}(\boldsymbol{x})|=|\phi_{\epsilon}(\boldsymbol{x})-\mathbf{f}(i)|=o_{\epsilon}(1)\;.\label{et3}
\end{equation}
By combining \eqref{et1}, \eqref{et2}, and \eqref{et3}, we can deduce
that
\begin{equation}
\mathbb{E}_{\boldsymbol{x}}^{\epsilon}\left[\phi_{\epsilon}(\boldsymbol{x}_{\epsilon}(a\theta_{\epsilon}\wedge\tau_{\mathcal{E}\setminus\mathcal{E}_{i}}))\right]\le Ca+o_{\epsilon}(1)\;.\label{et4}
\end{equation}
Next we attempt to bound the last expectation from below. Let $\delta>0$
be an arbitrarily small constant. Then, again by \eqref{poi2}, we
have $\phi_{\epsilon}+\delta\ge0$ on $\mathcal{E}_{i}$ and $\phi_{\epsilon}+\delta\ge1$
on $\mathcal{E}\setminus\mathcal{E}_{i}$ for all sufficiently small
$\epsilon$. Hence, by the maximum principle, $\phi_{\epsilon}+\delta\ge0$
on $\Omega$. Summing these, we have $\phi_{\epsilon}+\delta\ge\chi_{\mathcal{E}\setminus\mathcal{E}_{i}}$.
Therefore,
\begin{equation}
\mathbb{E}_{\boldsymbol{x}}^{\epsilon}\left[\phi_{\epsilon}(\boldsymbol{x}_{\epsilon}(a\theta_{\epsilon}\wedge\tau_{\mathcal{E}\setminus\mathcal{E}_{i}}))+\delta\right]\ge\mathbb{E}_{\boldsymbol{x}}^{\epsilon}\left[\chi_{\mathcal{E}\setminus\mathcal{E}_{i}}(\boldsymbol{x}_{\epsilon}(a\theta_{\epsilon}\wedge\tau_{\mathcal{E}\setminus\mathcal{E}_{i}}))\right]=\mathbb{P}_{\boldsymbol{x}}^{\epsilon}\left[\tau_{\mathcal{E}\setminus\mathcal{E}_{i}}<a\theta_{\epsilon}\right]\;.\label{et5}
\end{equation}
By \eqref{et4} and \eqref{et5}, we deduce
\[
\limsup_{\epsilon\rightarrow0}\sup_{\boldsymbol{x}\in\mathcal{E}_{i}}\mathbb{P}_{\boldsymbol{x}}^{\epsilon}\left[\tau_{\mathcal{E}\setminus\mathcal{E}_{i}}\le a\theta_{\epsilon}\right]\le Ca+\delta
\]
Because $\delta>0$ is arbitrary, the proof is completed.
\end{proof}
Next, we discuss the second ingredient. Denote by $\Delta=\Omega\setminus\mathcal{E}$
the outside of the metastable sets. Subsequently, by \eqref{inv1} and \eqref{inv2},
we have $\mu_{\epsilon}(\Delta)=o_{\epsilon}(1)$. In other words,
the set $\Delta$ is negligible in view of the equilibrium measure.
However, the second ingredient of tightness requires us to show this
negligibility of $\Delta$ in a dynamical sense. Hence, we define
the excursion time of the process $\boldsymbol{x}_{\epsilon}(t)$
on the set $\Delta$ up to time $t$ as
\[
\Delta(t)=\int_{0}^{t}\chi_{\Delta}(\boldsymbol{x}_{\epsilon}(s))\,ds\;.
\]
We remark that $\Delta(t)$ is a notion depending on $\epsilon$ although
we did not stress this in the notation. Then, we can formulate the
dynamic negligibility of $\Delta$ as follows:
\begin{lem}
\label{p34}For any sequence $(\pi_{\epsilon})_{\epsilon>0}$ of probability
measures concentrated on $\mathcal{E}_{i}$, the following holds:
\[
\lim_{\epsilon\rightarrow0}\theta_{\epsilon}^{-1}\mathbb{E}_{\pi_{\epsilon}}^{\epsilon}\left[\Delta(\theta_{\epsilon}t)\right]=0\;\;\text{for all }t\ge0.
\]
\end{lem}

Here, we only provide the proof of Lemma \ref{p34} when $\pi_{\epsilon}$
has a density function with respect to $\mu_{\epsilon}$ for each
$\epsilon>0$, and this density function belongs to $L^{p}(\mu_{\epsilon})$
for some $p>1$ in a uniform manner, i.e.,
\begin{equation}
\limsup_{\epsilon\rightarrow0}\int_{\Omega}\left|\frac{d\pi_{\epsilon}}{d\mu_{\epsilon}}\right|^{p}d\mu_{\epsilon}=M<\infty.\label{clp}
\end{equation}
For this case with mild initial distribution, we can deduce a simple
proof. For the general case, see the remark after the proof.
\begin{proof}[Proof of Lemma \ref{p34} under the assumption \eqref{clp}]
We fix $t\ge0$ and write
\[
u_{\epsilon}(\boldsymbol{x})=\theta_{\epsilon}^{-1}\mathbb{E}_{\boldsymbol{x}}^{\epsilon}\left[\Delta(\theta_{\epsilon}t)\right]\;.
\]
By Fubini's theorem, we obtain
\begin{equation}
\int_{\Omega}u_{\epsilon}\,d\mu_{\epsilon}=\theta_{\epsilon}^{-1}\mathbb{E}_{\mu_{\epsilon}}^{\epsilon}\left[\int_{0}^{\theta_{\epsilon}t}\chi_{\Delta}(\boldsymbol{x}_{\epsilon}(s))ds\right]=\theta_{\epsilon}^{-1}\int_{0}^{\theta_{\epsilon}t}\mathbb{P}_{\mu_{\epsilon}}^{\epsilon}\left[\boldsymbol{x}_{\epsilon}(s)\in\Delta\right]ds=t\,\mu_{\epsilon}(\Delta)\;.\label{vt1}
\end{equation}
We write $f_{\epsilon}=\frac{d\pi_{\epsilon}}{d\mu_{\epsilon}}$ so
that we can write
\begin{equation}
\theta_{\epsilon}^{-1}\mathbb{E}_{\pi_{\epsilon}}^{\epsilon}\left[\Delta(\theta_{\epsilon}t)\right]=\int_{\Omega}(u_{\epsilon}f_{\epsilon})\,d\mu_{\epsilon}\label{vt2}
\end{equation}
Now, we apply Holder's inequality, trivial bound $u_{\epsilon}\le t$,
\eqref{vt1} and \eqref{clp} to the right-hand side of the previous
display to deduce the following:
\begin{align}
\int_{\Omega}(u_{\epsilon}f_{\epsilon})\,d\mu_{\epsilon} & \le\left[\int_{\Omega}u_{\epsilon}\,d\mu_{\epsilon}\right]^{1/q}\left[\int_{\Omega}(u_{\epsilon}f_{\epsilon}^{p})\,d\mu_{\epsilon}\right]^{1/p}\nonumber \\
 & \le t\mu_{\epsilon}(\Delta)^{1/q}\,\left[\int_{\mathbb{R}^{d}}f_{\epsilon}^{p}\,d\mu_{\epsilon}\right]^{1/p}<t\mu_{\epsilon}(\Delta)^{1/q}M^{1/p}\;,\label{vt3}
\end{align}
where $q$ is the conjugate exponent of $p$ satisfying $\frac{1}{p}+\frac{1}{q}=1$.
Since $\mu_{\epsilon}(\Delta)=o_{\epsilon}(1)$, we complete the proof
by conditions \eqref{vt2} and \eqref{vt3}.
\end{proof}
\begin{rem}
To address the general case, it is sufficient to demonstrate that
\[
\lim_{\epsilon\rightarrow0}\sup_{\boldsymbol{x}\in\mathcal{E}_{i}}\theta_{\epsilon}^{-1}\mathbb{E}_{\boldsymbol{x}}^{\epsilon}\left[\Delta(\theta_{\epsilon}t)\right]=0\;.
\]
In view of the previous lemma for the special case, it suffices to
verify that
\[
\sup_{\boldsymbol{x},\,\boldsymbol{y}\in\mathcal{E}_{i}}\left(\mathbb{E}_{\boldsymbol{x}}^{\epsilon}\left[\Delta(\theta_{\epsilon}t)\right]-\mathbb{E}_{\boldsymbol{y}}^{\epsilon}\left[\Delta(\theta_{\epsilon}t)\right]\right)\ll\theta_{\epsilon}\;.
\]
When $\Omega$ is a discrete set, this can be typically performed
by the coupling (see \cite{BL1, BL2}). In the case of diffusion processes
in a $1$-dimensional torus, the same type of coupling argument can
be applied (see \cite{LS3}). However, the diffusion processes such
as the SRPDS \eqref{e12} in dimension $d\ge2$, two processes starting
from different points cannot be coupled exactly; hence, another argument
is required. In \cite[Appendix]{RS}, an argument based on the large-deviation
theory has been introduced. We refer to these listed articles for
the details of the proof of Lemma \ref{p34} for general cases.
\end{rem}

The proof of Proposition \ref{tight} from these two lemmas above
are routine and follows from Aldous' criterion. We refer to \cite[Section 7]{LS3}
or \cite[Section 5]{RS} for more details, but we herein provide a
brief sketch of the proof. First, we should introduce the appropriate
filtration. Write $\widehat{\boldsymbol{x}}_{\epsilon}(t)=\boldsymbol{x}_{\epsilon}(\theta_{\epsilon}t)$
as the accelerated process, and write $\widehat{\mathbb{P}}_{\pi_{\epsilon}}^{\epsilon}$
the law of process $(\widehat{\boldsymbol{x}}_{\epsilon}(t))_{t\ge0}$
starting from $\pi_{\epsilon}$. Then, denote by $\{\mathscr{F}_{t}^{0}:t\ge0\}$
the natural filtration of $D([0,\,\infty),\,\Omega)$ (or $C([0,\,\infty),\,\Omega)$
if $(\boldsymbol{x}_{\epsilon}(t))$ is a diffusion process) with
respect to the process $\widehat{\boldsymbol{x}}_{\epsilon}(\cdot)$,
namely,
\[
\mathscr{F}_{t}^{0}=\sigma(\widehat{\boldsymbol{x}}_{\epsilon}(s):s\in[0,\,t])\;.
\]
and define $\{\mathscr{F}_{t}:t\ge0\}$ as the usual augmentation
of $\{\mathscr{F}_{t}^{0}:t\ge0\}$ with respect to $\mathbb{\widehat{\mathbb{P}}}_{\pi_{\epsilon}}^{\epsilon}$
. Define $\mathscr{G}_{t}=\mathscr{F}_{S_{\epsilon}(t)}$ for $t\ge0$,
where $S_{\epsilon}$ is defined in \eqref{Se}. For $M>0$, we define
$\mathcal{\mathscr{T}}_{M}$ as the collection of stopping times with
respect to the filtration $\{\mathscr{G}_{t}\}_{t\ge0}$ bounded by
$M$.
\begin{proof}[Proof of Proposition \ref{tight}]
We start by considering the first statement of the proposition. By
Aldous' criterion, it suffices to verify that, for all $M>0$,
\begin{equation}
\lim_{a_{0}\rightarrow0}\limsup_{\epsilon\rightarrow0}\sup_{\tau\in\mathcal{\mathscr{T}}_{M}}\sup_{a\in(0,\,a_{0})}\mathbb{P}_{\pi_{\epsilon}}^{\epsilon}\left[\mathbf{x}_{\epsilon}(\tau+a)\neq\mathbf{x}_{\epsilon}(\tau)\right]=0\;.\label{e341}
\end{equation}
By Lemma \ref{p34}, it suffices to demonstrate that
\[
\lim_{a_{0}\rightarrow0}\limsup_{\epsilon\rightarrow0}\sup_{\tau\in\mathcal{\mathscr{T}}_{M}}\sup_{a\in(0,\,a_{0})}\mathbb{P}_{\pi_{\epsilon}}^{\epsilon}\left[\mathbf{x}_{\epsilon}(\tau+a)\neq\mathbf{x}_{\epsilon}(\tau),\,S_{\epsilon}(\tau+a)-S_{\epsilon}(\tau)\le2a_{0}\right]=0\;.
\]
Since $\mathbf{x}_{\epsilon}(t)=\Psi(\widehat{\boldsymbol{x}}_{\epsilon}(S_{\epsilon}(t)))$,
the last probability can be bounded above by
\[
\mathbb{P}_{\pi_{\epsilon}}^{\epsilon}\left[\Psi(\widehat{\boldsymbol{x}}_{\epsilon}(S_{\epsilon}(\tau)+t)\neq\Psi(\widehat{\boldsymbol{x}}_{\epsilon}(S_{\epsilon}(\tau)))\;\text{for some }t\in(0,\,2a_{0}]\right]\;.
\]
One can readily demonstrate that $S^{\epsilon}(\tau)$ is a stopping
time with respect to the filtration $\{\mathscr{F}_{t}\}$ (see \cite[Lemma 7.2]{LS3});
hence, by the strong Markov property the last probability is bounded
above by
\[
\sup_{\boldsymbol{x}\in\mathcal{E}_{i}}\mathbb{P}_{\boldsymbol{x}}^{\epsilon}\left[\Psi(\widehat{\boldsymbol{x}}_{\epsilon}(t))\neq\Psi(\boldsymbol{x})\;\text{for some }t\in(0,\,2a_{0}]\right]=\sup_{\boldsymbol{x}\in\mathcal{E}_{i}}\mathbb{P}_{\boldsymbol{x}}^{\epsilon}\left[\tau_{\breve{\mathcal{E}}_{i}}\le2a_{0}\theta_{\epsilon}\right]\;.
\]
Thus, the proof of \eqref{e341} is completed by Lemma \ref{tight2}.

The assertion $\mathbf{Q}^{*}(\mathbf{x}(0)=i)=1$ is trivial. For
the last assertion of the proposition, it suffices to prove that
\[
\lim_{a_{0}\rightarrow0}\limsup_{\epsilon\rightarrow0}\mathbb{P}_{\pi_{\epsilon}}^{\epsilon}\left[\mathbf{x}_{\epsilon}(t-a)\neq\mathbf{x}_{\epsilon}(t)\text{ for some }a\in(0,\,a_{0})\right]=0\;.
\]
The proof for this is the same as that above.
\end{proof}

\subsection{\label{sec23}Identification of limit points and the proof of Theorem
\ref{tmain}}

The proof of Theorem \ref{tmain} can be completed by identifying
the limit points of $(\mathbf{Q}_{\pi_{\epsilon}}^{\epsilon})_{\epsilon>0}$.
To this end, fix $\mathbf{f}\in\mathbb{R}^{S}$, and let $\phi_{\epsilon}=\phi_{\epsilon}^{\mathbf{f}}$
be the function obtained in Proposition \ref{poisson}. We fix $i$
and $\pi_{\epsilon}$ appearing in the statement of Theorem \ref{tmain}.
Since the generator of the accelerated process $\widehat{\boldsymbol{x}}(t)$
is $\theta_{\epsilon}\mathcal{L}_{\epsilon}$, the process $(M_{\epsilon}(t))_{t\ge0}$
defined by
\[
M_{\epsilon}(t)=\phi_{\epsilon}(\widehat{\boldsymbol{x}}_{\epsilon}(t))-\theta_{\epsilon}\int_{0}^{t}(\mathscr{\mathcal{L}}_{\epsilon}\phi_{\epsilon})(\widehat{\boldsymbol{x}}_{\epsilon}(s))ds
\]
is a martingale with respect to the filtration $\{\mathscr{F}_{t}\}$
introduced above. Since $\widehat{\boldsymbol{x}}_{\epsilon}(S_{\epsilon}(t))=\boldsymbol{\overline{x}}_{\epsilon}(\theta_{\epsilon}t)$
by definition, we can write
\begin{equation}
M_{\epsilon}(S_{\epsilon}(t))=\phi_{\epsilon}(\boldsymbol{\overline{x}}_{\epsilon}(\theta_{\epsilon}t))-\theta_{\epsilon}\int_{0}^{t}(\mathscr{\mathcal{L}}_{\epsilon}\phi_{\epsilon})(\boldsymbol{\overline{x}}_{\epsilon}(\theta_{\epsilon}s))ds\;.\label{ek0}
\end{equation}
Since $\mathscr{G}_{t}=\mathscr{F}_{S_{\epsilon}(t)}$ (see \cite[Lemma 5.5]{RS}),
the process $(M_{\epsilon}(S_{\epsilon}(t)))_{t\ge0}$ is a martingale
with respect to $\{\mathscr{G}_{t}\}$. By Proposition \ref{poisson},
we have
\[
\phi_{\epsilon}=\mathbf{f}\circ\Psi+o_{\epsilon}(1)\;\;\text{and\;\;}\mathscr{\theta_{\epsilon}\mathcal{L}}_{\epsilon}\phi_{\epsilon}=(\mathbf{L}\mathbf{f})\circ\Psi+o_{\epsilon}(1)\;\;\text{on }\mathcal{E}\;.
\]
Since $\Psi(\boldsymbol{\overline{x}}_{\epsilon}(\theta_{\epsilon}t))=\mathbf{x}_{\epsilon}(t)$,
we can rewrite \eqref{ek0} as
\begin{align*}
M_{\epsilon}(S_{\epsilon}(t))) & =\mathbf{f}(\mathbf{x}_{\epsilon}(t))-\int_{0}^{t}(\mathbf{L}\mathbf{f})(\mathbf{x}_{\epsilon}(s))ds+o_{\epsilon}(1)\;.
\end{align*}
Hence, if $\mathbf{Q}^{*}$ is a limit point of the family $\{\mathbf{Q}_{\pi_{\epsilon}}^{\epsilon}\}_{\epsilon\in(0,\,1]}$,
the process $(M^{*}(t))_{t\ge0}$ defined by
\begin{equation}
M^{*}(t)=\mathbf{f}(\mathbf{x}(t))-\int_{0}^{t}(\mathbf{L}\mathbf{f})(\mathbf{x}(s))ds\label{mt}
\end{equation}
is a martingale under $\mathbf{Q}^{*}$. This completes the characterization
of the limit points.
\begin{proof}[Proof of Theorem \ref{tmain}]
 Let $\mathbf{Q}^{*}$ be a limit point of the family $\{\mathbf{Q}_{\pi_{\epsilon}}^{\epsilon}\}_{\epsilon\in(0,\,1]}$.
Subsequently, as demonstrated above, $(M(t))_{t\ge0}$ defined by
\eqref{mt} is a martingale under $\mathbf{Q}^{*}$; furthermore,
by Proposition \ref{tight}, we obtain $\mathbf{Q}^{*}[\mathbf{x}(0)=i]=1$
and $\mathbf{Q}^{*}(\mathbf{x}(t)\neq\mathbf{x}(t-))=0$ for all $t>0$.
The only probability measure on $D([0,\,\infty),\,S)$ satisfying
these properties simultaneously is $\mathbf{Q}_{i}$; thus, we can
conclude that $\mathbf{Q}^{*}=\mathbf{Q}_{i}$. This completes the
proof.
\end{proof}

\subsection{\label{sec33}Analysis of Poisson Equation }

In view of the argument presented in the previous section, the entire
analysis is dependent on the construction of $\phi_{\epsilon}^{\mathbf{f}}$
introduced in Proposition \ref{poisson}. This should be proven for
each model.

For the non-reversible diffusion processes in a $1$-dimensional torus
considered in \cite{LS3}, we found an explicit solution $\psi_{\epsilon}$
of the equation \eqref{poi1}. It is noteworthy that for each constant
$c_{\epsilon}$, the function $\psi_{\epsilon}+c_{\epsilon}$ is also
a solution of \eqref{poi1}. Hence, in this case we can select $c_{\epsilon}$
carefully such that the function $\phi_{\epsilon}=\psi_{\epsilon}+c_{\epsilon}$
satisfies \eqref{poi2} as well. We refer to \cite[Section 9]{LS3}
for the details.

For the SRPDS \eqref{e12} in $\mathbb{R}^{d}$, we cannot expect
such a closed form solution. We shall sketch our idea of the proof
in the next paragraph. We believe that the idea presented herein can
be applied to a broad scope of examples as well after a suitable modification.
We refer to \cite[Section 4]{RS} for the full details.

Let $\mathcal{D}_{\epsilon}(\phi)=\int_{\Omega}\phi(-\mathcal{L}_{\epsilon}\phi)d\mu_{\epsilon}$
be the Dirichlet form associated to the process $\boldsymbol{y}_{\epsilon}(t)$
defined in \eqref{e12}. Since $\mathcal{L}_{\epsilon}$ is self-adjoint
with respect to $d\mu_{\epsilon}$, we can find a solution of the
Poisson equation \eqref{poi1} by a minimizer of the functional defined
by
\[
\mathcal{I}_{\epsilon}(\phi)=\frac{1}{2}\theta_{\epsilon}\,\mathcal{D}_{\epsilon}(\phi)+\sum_{i\in S}\mathbf{a}_{\epsilon}(i)\,(\mathbf{L}\mathbf{f})(i)\,\int_{\mathcal{E}_{i}}\phi\,d\mu_{\epsilon}\;.
\]
Take a minimizer $\psi_{\epsilon}$ of this functional so that $\psi_{\epsilon}$
solves the equation \eqref{poi1}. For this minimizer, one can readily
show that, for some $\lambda_{\epsilon}>0$,
\[
\theta_{\epsilon}\,\mathcal{D}_{\epsilon}(\psi_{\epsilon})=\lambda_{\epsilon}\;\;\text{and}\;\;\sum_{i\in S}\mathbf{a}_{\epsilon}(i)\,(\mathbf{L}\mathbf{f})(i)\,\int_{\mathcal{E}_{i}}\psi_{\epsilon}\,d\mu_{\epsilon}=-\lambda_{\epsilon}\;.
\]
Then, one can show that $\lambda_{\epsilon}=O_{\epsilon}(1)$. Denote
by $m$ the Lebesgue measure on $\mathbb{R}^{d}$ and define $\mathbf{q}=(\mathbf{q}(i):i\in S)\in\mathbb{R}^{S}$
such a manner that
\[
\mathbf{q}(i)=\frac{1}{m(\mathcal{E}_{i})}\int_{\mathcal{E}_{i}}\psi_{\epsilon}(\boldsymbol{x})d\boldsymbol{x}\;\;;\;i\in S\;.
\]
Subsequently, by the bound on $\mathcal{D}_{\epsilon}(\psi_{\epsilon})$
and Poincare's inequality we can verify that
\begin{equation}
\Vert\psi_{\epsilon}-\mathbf{q}(i)\Vert_{L^{2}(\mathcal{E}_{i})}=\int_{\mathcal{E}_{i}}(\psi_{\epsilon}-\mathbf{q}(i))^{2}dx=o_{\epsilon}(1)\;\;\text{for all }i\in S.\label{ee1}
\end{equation}
Then, a technique developed in \cite{ET, ST} based on the interior
elliptic estimate in the partial differential equations theory allows
us to prove that
\begin{equation}
\Vert\psi_{\epsilon}-\mathbf{q}(i)\Vert_{L^{\infty}(\mathcal{E}_{i})}=\lambda_{\epsilon}o_{\epsilon}(1)\;\;\text{for all }i\in S.\label{ee2}
\end{equation}
To obtain this, we first prove \eqref{ee1} for a slightly larger
set $\widetilde{\mathcal{E}}_{i}\Supset\mathcal{E}_{i}$ and then
use the interior elliptic estimate \cite[Theorem 8.17]{GT} to enhance
the $L^{2}$-estimate to the interior $L^{\infty}$-estimate. For
the detail we refer to \cite[Proposition 4.8]{RS}.

The final step is to find a constant $c_{\epsilon}$ such that $\mathbf{q}(i)+c_{\epsilon}=\mathbf{f}(i)$
for all $i\in S$. We will not provide the detailed proof for this,
but we strongly recommend the readers to read \cite[Section 4.5]{RS},
in which a novel method to prove the characterization of $\mathbf{q}$
has been developed. The primary idea is to couple the function $\psi_{\epsilon}$
with a test function that is already popular in the study of metastability.
We believe that the argument therein can be applied to a broad scope
of models.

\section{\label{sec4}Conclusion}

At the time when this review paper was written, the approach via the
Poisson equation introduced herein has been applied to the reversible
SRPDS \eqref{e12} in \cite{RS}, and the non-reversible SRPDS in
a $1$-dimensional torus in \cite{LS3}. We believe that this methodology
can be applied to a wide range of models exhibiting metastability.
In particular, the approach based on the Poisson equation did not
heavily use the reversibility of the underlying metastable process;
hence, we hope that our approach paves the way for the quantitative
analysis of non-reversible metastable processes, in which many open
problems still remain.

\end{document}